\documentclass{amsart}
\usepackage{enumerate,mathtools,amssymb, enumitem, tikz-cd, mathabx}
\usepackage[hide links]{hyperref}

\theoremstyle{plain}
\newtheorem{theorem}{Theorem}[section]
\newtheorem{lemma}[theorem]{Lemma}
\newtheorem{proposition}[theorem]{Proposition}
\newtheorem{conjecture}[theorem]{Conjecture}

\theoremstyle{definition}
\newtheorem{definition}[theorem]{Definition}
\newtheorem{remark}[theorem]{Remark}
\newtheorem*{acknowledgements}{Acknowledgements}

\numberwithin{equation}{section}

\begin{document}

\title[MF approximations of amalgamated products of groups]{Finite dimensional approximations of certain amalgamated free products of groups}
\author{Christopher Schafhauser}
\address{Department of Mathematics, University of Nebraska-Lincoln, Lincoln, NE, USA}
\email{cschafhauser2@unl.edu}
\date{\today}
\subjclass[2020]{46L05}

\begin{abstract}
	A group is called matricial field (MF) if it admits finite dimensional approximate unitary representations which are approximately faithful and approximately contained in the left regular representation.  This paper provides a new class of MF groups by showing that given two amenable groups with a common normal subgroup, the amalgamated free product is MF.
\end{abstract}

\maketitle

\section{Introduction}

Some of the most prominent classes of operator algebras arise from groups and group representations.  As far back as the introductory paper on operator algebras~(\cite{Murray-vonNeumann36}), Murray and von Neumann discuss the possibility of developing a representation theory for infinite groups as one of their motivations for investigating rings of operators.  The left regular representation of a (discrete) group $G$, given by
\begin{equation}
	\lambda^G \colon G \rightarrow U(\ell^2(G)) \colon g \mapsto (\delta_h \mapsto \delta_{gh}), \qquad g, h \in G,
\end{equation}
is of particular interest.  In fact, the reduced group $C^*$-algebra $C^*_\lambda(G)$ and the group von Neumann algebra $L(G)$, which are, respectively, the norm and weak closed span of the image of $\lambda^G$, are some of the most studied operator algebras.

In both group theory and operator algebras, approximation properties are ubiquitous.  For example, amenability, which can be viewed as an internal approximation of the dynamics by finite (or finite dimensional) objects, has a fundamental role in both subjects.  A much weaker approximation property is given by the notion of hyperlinearity for groups or the corresponding condition for von Neumann algebras: embeddability into $R^\omega$, the tracial ultrapower of the hyperfinite II$_1$-factor (see \cite{Ozawa04} for a survey).  In fact, a countable group $G$ is hyperlinear if and only if $L(G)$ embeds into $R^\omega$ in a trace-preserving way; see \cite[Theorem~8.5]{Pestov08}, for example.  The Connes embedding problem, asking if every (separably acting) tracial von Neumann algebra embeds into $R^\omega$, has been a central problem in operator algebras since shortly after its inception in \cite{Connes76}.  A negative solution was recently announced in \cite{JNVWY}; it is still an open question if all groups are hyperlinear.

The present paper is concerned with an approximation property for groups motivated by work in $C^*$-algebras---particularly, the MF $C^*$-algebras from \cite{Blackadar-Kirchberg97}.  The MF property is variation of hyperlinearity obtained by equipping the matrix unitary groups with a more rigid metric. Roughly, a group is MF if the group admits finite dimensional approximate unitary representations which are approximately faithful and approximately contained in the left regular representation.

For all $d \in \mathbb N$, let $U(d)$ denote the group of $d \times d$ unitary matrices, let $\|\cdot\|$ denote the operator norm on $M_d$, the algebra of $d \times d$ matrices over $\mathbb C$, and let $\mathrm{tr}_d$ denote the trace on $M_d$ with the normalization $\mathrm{tr}_d(1) = 1$.  
\begin{definition}[cf.\ \cite{Blackadar-Kirchberg97}]\label{def:MF-group}
	A group $G$ is \emph{matricial field (MF)} if for all finite sets $\mathcal G \subseteq G$ and $\epsilon > 0$, there are $d \in \mathbb N$ and a function $u \colon G \rightarrow U(d)$ with
	\begin{enumerate}
		\item\label{def:MF-group1} $\|u_{gh} - u_gu_h\| < \epsilon$ for all $g, h \in \mathcal G$,
		\item\label{def:MF-group2} $|\mathrm{tr}_G(u_g)| < \epsilon$ for all $g \in \mathcal G \setminus\{1\}$, and
 		\item\label{def:MF-group3} $\big|\big\| \sum\limits_{g \in \mathcal G} c_g u_g \big\| - \big\| \sum\limits_{g \in \mathcal G} c_g \lambda^G_g\big\| \big| < \epsilon$ for all $(c_g)_{g \in \mathcal G} \subseteq \mathbb C$ with $\max\limits_{g \in \mathcal G} |c_g| \leq 1$.
	\end{enumerate}
\end{definition}

Note that the definition here is (a priori) quite a bit stronger than the definition of MF group in \cite{Carrion-Dadarlat-Eckhardt13}.  The above definition is more closely related to approximation properties of the $C^*$-algebra.  In fact, with the present terminology, a group $G$ is MF if and only if the canonical trace on $C^*_\lambda(G)$ is MF (see Proposition~\ref{prop:MF}).  

Every MF group is hyperlinear.  As with hyperlinearity, there are no known examples of groups which are not MF. The only known examples of non-MF traces on $C^*$-algebras are those not factoring through $R^\omega$, arising from \cite{JNVWY}.

The class of known examples of MF groups is rather small.  The easiest class of non-trivial examples is that of residually finite amenable groups.  Using residual finiteness, one can construct finite dimensional unitary representations of $G$ satisfying \ref{def:MF-group2} (and obviously \ref{def:MF-group1}) from the left regular representations of finite quotient groups, and amenability forces~\ref{def:MF-group3}.  As far as permanence properties, it is easy to see that MF is a local property which passes to subgroups and finite index supergroups.  With slightly more work, one can show that products of MF groups are MF provided that one factor is exact (see \cite[Proposition~3.11(iii)]{Hayes15}).

Much more surprising examples come from the celebrated results of Haagerup and Thorbj{\o}rnsen (\cite{Haagerup-Thorbjornsen05}) stating that free groups are MF and Tikuisis, White, and Winter (\cite{Tikuisis-White-Winter17}) stating that amenable groups are MF.  The Haagerup--Thorbj{\o}rnsen theorem was generalized by Hayes in \cite{Hayes15} (using \cite{Collins-Male14})
to show that free products of MF groups are MF.  Utilizing deep results from the classification theory of amenable $C^*$-algebras, Rainone and I showed in \cite{Rainone-Schafhauser19} that if $G$ is an amenable group and $F$ is a free group acting on $G$, then $G \rtimes F$ is MF (\cite[Theorem~1.3(i)]{Rainone-Schafhauser19}).  Very recent results of Louder and Magee (\cite{Louder-Magee22}) and Magee and Thomas (\cite{Magee-Thomas23}) show, respectively, that all limit groups and all right-angled Artin groups are MF.

At a 2018 BIRS workshop, Hayes posed the following problem, motivated by the results in the previous paragraph and the analogous theorems for hyperlinear and sofic groups---see \cite[Corollary~4.5]{Brown-Dykema-Jung08} for hyperlinearity (stated in terms of von Neumann algebras) and \cite{Elek-Szabo11, Paunescu11} for soficity.

\begin{conjecture}[Hayes]\label{conj}
	If $G$ and $H$ are MF groups and $A$ is a common amenable subgroup, then $G *_A H$ is MF.
\end{conjecture}

The present paper proves the following special case of Conjecture~\ref{conj}.

\begin{theorem}\label{thm:main}
	If $G$ and $H$ are amenable groups with a common normal (necessarily amenable) subgroup $N$, then $G *_N H$ is MF.
\end{theorem}

The main tool is a classification result which follows from \cite{Schafhauser20}.  Roughly, if $G$ is an amenable group, then up to unitary equivalence, matrix amplification, and norm perturbation, there is a unique function $u \colon G\rightarrow U(d)$ satisfying the conditions of Definition~\ref{def:MF-group} (see Theorem~\ref{thm:af-emb}). This will be used to construct finite dimensional approximate unitary representations of $G *_N H$, which is then combined with MF approximations for the quotient to produce MF approximations in Theorem~\ref{thm:main}.

\begin{acknowledgements}
The work in this paper was partially supported by  NSF Grant DMS-2000129.  I am grateful to Ben Hayes for helpful conversations related to the results of this paper.  I also thank the referee for helpful feedback on this paper and Michael Magee for drawing my attention to \cite{Louder-Magee22} and \cite{Magee-Thomas23}.
\end{acknowledgements}

\section{Notation and Preliminaries}

A trace on a $C^*$-algebra will always mean a tracial state.  The $C^*$-algebra of $d \times d$ matrices over $\mathbb C$ is denoted $M_d$, and  the unique trace on this algebra is written $\mathrm{tr}_d$.  For a $C^*$-algebra $A$, write $U(A)$ for the unitary group of $A$.  If $G$ is a discrete group, denote the full and reduced group $C^*$-algebras of $G$ by $C^*(G)$ and $C^*_\lambda(G)$, respectively.  The canonical unitary representations of $G$ on $C^*(G)$ and $C^*_\lambda(G)$ will be denoted by $u^G$ and $\lambda^G$, respectively, and the quotient map from the full to the reduced group $C^*$-algebra will be denoted by $\pi_G \colon C^*(G) \rightarrow C^*_\lambda(G)$, so $\pi_G(u^G_g) = \lambda^G_g$ for all $g \in G$. Let $\mathrm{tr}_G$ denote trace on $C^*_\lambda(G)$ determined by $\mathrm{tr}_G(\lambda^G_g) = 0$ for all $g \in G \setminus\{1\}$.

The notion of an MF trace was introduced in \cite{Rainone-Schafhauser19} following the approximation property characterization of MF $C^*$-algebras from \cite{Blackadar-Kirchberg97} and the approximation properties of traces (especially quasidiagonality) introduced in \cite{Brown06}.

\begin{definition}\label{def:MF-trace}
A trace $\mathrm{tr}_A$ on a $C^*$-algebra $A$ is called \emph{matrial field (MF)} if for all finite sets $\mathcal F \subseteq A$ and $\epsilon > 0$, there are $d \in \mathbb N$ and a self-adjoint linear map $\phi \colon A \rightarrow M_d$ such that for all $a, b \in \mathcal F$,
\begin{equation} \|\phi(ab) - \phi(a)\phi(b)\| < \epsilon \qquad \text{and} \qquad  
|\mathrm{tr}_d(\phi(a)) - \mathrm{tr}_A(a)| < \epsilon. \end{equation}
\end{definition}

It will often be convenient (both notationally and technically) to replace the matrix algebras $M_d$ in the above definition with the UHF algebra $\mathcal Q = \bigotimes_{d = 1}^\infty M_d$.  Note that $\mathcal Q$ has a unique trace $\mathrm{tr}_\mathcal Q$, which is induced by the traces $\mathrm{tr}_d$.  The technical advantages of working with $\mathcal Q$ come from the extra divisibility which isn't present in finite dimensional matrix algebras $M_d$.  This is particularly important in the classification result from \cite{Schafhauser20} (Theorem~\ref{thm:af-emb} below).

In Definition~\ref{def:MF-trace}, replacing $M_d$ and $\mathrm{tr}_d$ with $\mathcal Q$ and $\mathrm{tr}_\mathcal Q$ does not change the definition.  Indeed, in one direction, one may compose a map into $M_d$ with a unital embedding $M_d \hookrightarrow \mathcal Q$, and in the other direction, one may compose a map into $\mathcal Q$ with a conditional expectation onto a sufficiently large matrix subalgebra of $\mathcal Q$.  In a similar fashion, the groups $U(d)$ in Definition~\ref{def:MF-group} can be replaced with the group $U(\mathcal Q)$ using that unitaries in $\mathcal Q$ can be approximated by unitaries in matrix subalgebras of $\mathcal Q$.

The following result justifies the claim made in the introduction on the relationship between the two notions of MF in Definitions~\ref{def:MF-group} and~\ref{def:MF-trace}.  The proof comes down to standard manipulation of approximate morphisms.

\begin{proposition}\label{prop:MF}
	A group $G$ is MF if and only if the trace $\mathrm{tr}_G$ is MF.
\end{proposition}

\begin{proof}
As both notions of MF are local properties, it suffices to prove the result when $G$ is countable.  Let $\omega$ be a free ultrafilter of $\mathbb N$, let $\mathcal Q_\omega$ be the corresponding norm ultrapower of $\mathcal Q$, and let $\mathrm{tr}_{\mathcal Q_\omega}$ be the trace on $\mathcal Q_\omega$ induced by $\mathrm{tr}_\mathcal Q$.

If $\mathrm{tr}_G$ is MF, then there is a unital $^*$-homomorphism $\phi \colon C^*_\lambda(G) \rightarrow \mathcal Q_\omega$ with $\mathrm{tr}_{\mathcal Q_\omega} \circ \phi = \mathrm{tr}_G$.  Its worth noting that there is a subtlety here arising from the fact that the maps in Definition~\ref{def:MF-trace} are not assumed to be bounded or unital.  However, it's possible to arrange the maps $\phi$ in the definition to be ``approximately bounded'' in a suitable sense, which is sufficient to produce a (possibly non-unital) $^*$-homomorphism $\phi_0$ with $\mathrm{tr}_{\mathcal Q_\omega} \phi_0 = \mathrm{tr}_G$; see \cite[Proposition~2.3]{Rainone-Schafhauser19} for details.  Since any non-zero corner of $\mathcal Q_\omega$ is isomorphic to $\mathcal Q_\omega$ (see \cite[Proposition~1.3(i)]{Tikuisis-White-Winter17}, for example), the desired map $\phi$ is given by composing $\phi_0$ with an isomorphism $\phi_0(1) \mathcal Q_\omega \phi_0(1) \cong \mathcal Q_\omega$, which is necessarily trace-preserving by the uniqueness of the trace on $\mathcal Q_\omega$ (see \cite[Theorem~8]{Ozawa13}, for example).

Define $u \colon G \rightarrow U(\mathcal Q_\omega)$ by $u_g = \phi(\lambda_g^G)$.  As every unitary in $\mathcal Q_\omega$ is represented by a sequence of unitaries in $\mathcal Q$, there is a sequence of functions $(u^k \colon G \rightarrow U(\mathcal Q))_{k=1}^\infty$ such that $(u^k_g)_{k=1}^\infty$ represents $u_g$ for all $g \in G$.
Then 
\begin{align}
	\lim_{k \rightarrow \omega} \|u^k_{gh} - u^k_g u^k_h\| &= 0, \qquad g, h \in G,
\intertext{and}
	\lim_{k \rightarrow \omega} \mathrm{tr}_\mathcal Q (u^k_g) &= 0, \qquad g \in G \setminus \{1\}.
\end{align}
Further, as $\mathrm{tr}_G$ is faithful and $\mathrm{tr}_{\mathcal Q_\omega} \circ \phi = \mathrm{tr}_G$, we have that $\phi$ is faithful. Thus $\phi$ is isometric, being a faithful $^*$-homomorphism, and hence for all finite sets $\mathcal G \subseteq G$ and $(c_g)_{g \in \mathcal G} \subseteq \mathbb C$,
\begin{equation}\label{eq:approx-contained}
	\lim_{k \rightarrow \omega} \Big\| \sum_{g \in \mathcal G} c_g u^k_g \Big\| = \Big\| \sum_{g \in \mathcal G} c_g \lambda^G_g \Big\|.
\end{equation}
It follows that given a finite set $\mathcal G \subseteq G$ and $\epsilon > 0$, there is $k \in \mathbb N$ such that $u^k$ satisfies the conditions of Definition~\ref{def:MF-group} (with $U(\mathcal Q)$ in place of $U(d)$), using \eqref{eq:approx-contained} and the compactness of $\big\{(c_g)_{g \in \mathcal G} : \max\limits_{g \in \mathcal G} |c_g| \leq 1\big\}$ to realize \ref{def:MF-group3}. 

The converse has a similar flavor.  Let $(\mathcal G_k)_{k=1}^\infty$ be an increasing sequence of finite subsets of $G$ with union $G$, and for each $k \in \mathbb N$, fix a function $u^k \colon G \rightarrow U(\mathcal Q)$ as in Definition~\ref{def:MF-group} with $\mathcal G_k$ and $1/k$ in place of $\mathcal G$ and $\epsilon$.  Let $u \colon G \rightarrow U(\mathcal Q_\omega)$ denote the group homomorphism induced by the sequence $(u^k)_{k=1}^\infty$ and note that $\mathrm{tr}_{\mathcal Q_\omega}(u_g) = 0$ for all $g \in G \setminus\{1\}$.

There is a unital $^*$-homomorphism $\tilde \phi \colon C^*(G) \rightarrow \mathcal Q_\omega$ with $\tilde\phi(u^G_g) = u_g$.  The norm bounds arising from  Definition~\ref{def:MF-group}\ref{def:MF-group3} imply $\|\tilde\phi(a)\| \leq \|\pi_G(a)\|$ for all $a \in C^*(G)$, and therefore, $\tilde\phi = \phi \circ \pi_G$ for some  unital $^*$-homomorphism $\phi \colon C^*_\lambda(G) \rightarrow \mathcal Q_\omega$.  The trace condition on $u$ implies $\mathrm{tr}_{\mathcal Q_\omega} \circ \phi = \mathrm{tr}_G$.

If $(\phi_k \colon C^*_\lambda(G) \rightarrow \mathcal Q)_{k=1}^\infty$ is a sequence of self-adjoint linear maps lifting $\phi$, then for all $a, b \in C^*_\lambda(G)$,
\begin{equation}
	\lim_{k \rightarrow \omega} \|\phi_k(ab) - \phi_k(a) \phi_k(b)\| = 0 \quad \text{and} \quad \lim_{k \rightarrow \omega} \mathrm{tr}_\mathcal Q(\phi_k(a)) = \mathrm{tr}_G(a).
\end{equation}
It follows that given a finite set $\mathcal F \subseteq C^*_\lambda(G)$ and $\epsilon > 0$, there is $k \in \mathbb N$ such that $\phi_k$ satisfies the conditions of Definition~\ref{def:MF-trace} (with $\mathcal Q$ in place of $M_d$).
\end{proof}

The following three deep results form the core of the proof of Theorem~\ref{thm:main}.  The first is a essentially a restatement of the positive solution of Rosenberg's conjecture in \cite{Tikuisis-White-Winter17} stating that amenable group $C^*$-algebras are quasidiagonal.

\begin{theorem}[cf.\ {\cite[Corollary~C]{Tikuisis-White-Winter17}}]\label{thm:qd}
	Amenable groups are MF.
\end{theorem}

\begin{proof}
	The proof of \cite[Corollary~C]{Tikuisis-White-Winter17} shows that $\mathrm{tr}_G$ is quasidiagonal, which (by definition) means the conditions of Definition~\ref{def:MF-trace} hold with completely positive contractive maps $\phi$.  Indeed, we may assume $G$ is countable as MF is a local property.  In this case, $G$ is a-T-menable by the main result of \cite{Bekka-Cherix-Valette93}, and hence satisfies the UCT by \cite[Proposition~10.7]{Tu99}.  The claim follows from  \cite[Theorem~A]{Tikuisis-White-Winter17}.
\end{proof}

The second result, taken from \cite{Hayes15}, is a consequence of a large body of work in free probability, including \cite{Voiculescu91, Haagerup-Thorbjornsen05, Collins-Male14}

\begin{theorem}[{\cite[Proposition~3.11(v)]{Hayes15}}]\label{thm:free}
	Free products of MF groups are MF.
\end{theorem}

The following is a simple consequence of the classification results from \cite{Schafhauser20} after specializing to the case of unital embeddings $C^*_\lambda(G) \rightarrow \mathcal Q$.

\begin{theorem}[cf.\ {\cite[Thoerem~B]{Schafhauser20}}]\label{thm:classification}\label{thm:af-emb}
	If $G$ is a countable amenable group, then there is a group homomorphism $u \colon G \rightarrow U(\mathcal Q)$ such that $\mathrm{tr}_\mathcal Q(u_g) = 0$ for all $g \in G \setminus \{1\}$, and this $u$ is unique up to approximate unitary equivalence.
\end{theorem}

\begin{proof}
	By \cite[Theorem~B]{Schafhauser20}, there is a unital embedding $\phi \colon C^*_\lambda(G) \rightarrow \mathcal Q$ such that $\mathrm{tr}_\mathcal Q \circ \phi = \mathrm{tr}_G$.  Define $u_g = \phi(\lambda^G_g)$ for $g \in G$.  To see uniqueness, suppose that $v \colon G \rightarrow U(\mathcal Q)$ is another group homomorphism with $\mathrm{tr}_\mathcal Q(v_g) = 0$ for all $g \in G \setminus \{1\}$.  Since $G$ is amenable, there is a $^*$-homomorphism $\psi \colon C^*_\lambda(G) \rightarrow \mathcal Q$ with $\psi(\lambda_g^G) = v_g$ for all $g \in G$.  Note that $\mathrm{tr}_\mathcal Q \circ \psi = \mathrm{tr}_G$.
		
	It suffices to show that $\phi$ and $\psi$ are approximately unitarily equivalent.  The conditions of \cite[Corollary~5.4(2)]{Schafhauser20} apply (the UCT for the group algebra follows from \cite{Bekka-Cherix-Valette93, Tu99}, as in the proof of Theorem~\ref{thm:qd}) and  $\mathrm{tr}_\mathcal Q \circ \phi = \mathrm{tr}_\mathcal Q \circ \psi$, so it is enough to show that $K_0(\phi) = K_0(\psi)$.  Since the map $K_0(\mathcal Q) \rightarrow \mathbb R$ induced by $\mathrm{tr}_\mathcal Q$ is injective, the agreement on trace implies agreement on $K_0$.
\end{proof}

\section{The proof of Theorem~\ref{thm:main}}

The following lemma, motivated by \cite[Theorem~3.9]{Rainone-Schafhauser19}, gives a sufficient condition for a group to be MF.

\begin{lemma}\label{lemma}
Suppose $G$ is a group with a normal amenable subgroup $N$ such that $G / N$ is exact and MF.  Then $G$ is MF if and only if there are a unital \mbox{$C^*$-algebra} $A$ with an MF trace $\mathrm{tr}_A$ and a group homomorphism $u \colon G \rightarrow U(A)$ such that $\mathrm{tr}_A(u_n) = 0$ for all $n \in N \setminus \{1\}$. 
\end{lemma}

\begin{proof}
For the forward direction, take $A = C^*_\lambda(G)$ and $\mathrm{tr}_A = \mathrm{tr_G}$ (which is MF by Proposition~\ref{prop:MF}) and define $u_g = \lambda^G_g$.  To see the converse, let $q \colon G \rightarrow G/N$ denote the quotient map and let $\tilde\phi_0 \colon C^*(N) \rightarrow A$ and $\tilde \phi \colon C^*(G) \rightarrow A \otimes C^*_\lambda(G / N)$ be the $^*$-homomorphisms given by
\begin{equation}
	\tilde\phi_0(u_n^N) = u_n \quad \text{and} \quad \tilde\phi(u^G_g) = u_g \otimes \lambda^{G/N}_{q(g)}, \qquad n \in N,\ g \in G,
\end{equation}
where the tensor product is the spatial one.  Since $\mathrm{tr}_A$ and $\mathrm{tr}_{G/N}$ are MF and $C^*_\lambda(G/N)$ is exact, the trace $\mathrm{tr}_A \otimes \mathrm{tr}_{G/N}$ is MF by \cite[Proposition~3.6]{Rainone-Schafhauser19}, so it suffices to show that $\tilde\phi$ factors through a $^*$-homomorphism $\phi \colon C^*_\lambda(G) \rightarrow A \otimes C^*_\lambda(G/N)$ with $(\mathrm{tr}_A \otimes \mathrm{tr}_{G/N}) \circ \phi = \mathrm{tr}_G$.  Indeed, composing MF approximations of $\mathrm{tr}_A \otimes \mathrm{tr}_{G/N}$ with $\phi$ will produce MF approximations for $\mathrm{tr}_G$.

As $N$ is amenable, $\pi_N$ is an isomorphism, and hence $\tilde\phi_0$ factors as $\phi_0 \circ \pi_N$ for a $^*$-homomorphism $\phi_0 \colon C^*_\lambda(N) \rightarrow A$.
Consider the commuting diagram
\begin{equation}\label{eq:diagram}
\begin{tikzcd}[row sep = small, column sep = small]
	&[10pt]C^*_\lambda(G) \arrow{dd}[pos=.6]{E_\lambda} \arrow[dashed]{dr}[near start]{\phi} & \\[-3pt]
	C^*(G) \arrow[crossing over]{rr}[pos=.3, swap]{\tilde\phi} \arrow{dd}[swap]{E} \arrow{ur}[near end]{\pi_G} & & A \otimes C^*_\lambda(G/N) \arrow{dd}{\mathrm{id}_A \otimes \mathrm{tr}_{G/N}} \\[10pt]
	& C^*_\lambda(N) \arrow{dr}[near start]{\phi_0} & \\[-3pt]
	C^*(N) \arrow{rr}[swap]{\tilde\phi_0} \arrow{ur}[near end]{\pi_N} & & A
\end{tikzcd}
\end{equation}
where $E$ and $E_\lambda$ are the canonical conditional expectations.  As $\mathrm{id}_A$ and $\mathrm{tr}_{G/N}$ are faithful, so is $\mathrm{id}_A \otimes \mathrm{tr}_{G/N}$ (see the appendix of \cite{Avitzour82}, for example).  It follows that if $a \in \ker(\pi_G)$ is positive, then $\tilde\phi(a) = 0$.  Hence there is a $^*$-homomorphism $\phi$ making the upper face of \eqref{eq:diagram} commute.  The back right face also commutes by a diagram chase using the surjectivity of $\pi_G$ and the commutativity of the other four faces.  The commutativity of this face implies $\phi$ is trace preserving as $E_\lambda$ and $\phi_0$ are.
\end{proof}

The following is a consequence of the classification result in Theorem~\ref{thm:af-emb}. In the context of Theorem~\ref{thm:main}, this lemma provides MF approximations of $G$ and $H$ which approximately coincide on $N$.  Note that $N$ need not be normal.

\begin{lemma}\label{lem:nets}
	If $G$ and $H$ are amenable groups with a common subgroup $N$, then there are nets of functions $(u^i\colon G \rightarrow U(\mathcal Q))_{i \in I}$ and $(v^i \colon H \rightarrow U(\mathcal Q))_{i \in I}$ such that for all $g_1, g_2 \in G$, $g \in G \setminus \{1\}$, $h_1, h_2 \in H$, $h \in H \setminus \{1\}$, and $n \in N$,
	\begin{enumerate}
		\item $\lim\limits_{i \in I} \|u^i_{g_1g_2} - u^i_{g_1} u^i_{g_2} \| = 0$ and  $\lim\limits_{i \in I} \|v^i_{h_1 h_2} - v^i_{h_1} v^i_{h_2}\| = 0$,
		\item $\lim\limits_{i \in I} \mathrm{tr}_\mathcal Q(u^i_g) = 0$ and $\lim\limits_{i \in I} \mathrm{tr}_\mathcal Q(v^i_h) = 0$, and
		\item $\lim\limits_{i \in I} \|u^i_n - v^i_n\|= 0$.
	\end{enumerate}
\end{lemma}

\begin{proof}
	Let $I$ be the set of all triples $(\mathcal G, \mathcal H, \epsilon)$ where $\mathcal G \subseteq G$ and $\mathcal H \subseteq H$ are finite sets and $\epsilon > 0$ and write $(\mathcal G_1, \mathcal H_1, \epsilon_1) \leq (\mathcal G_2, \mathcal H_2, \epsilon_2)$ if $\mathcal G_1 \subseteq \mathcal G_2$, $\mathcal H_1 \subseteq \mathcal H_2$, and $\epsilon_1 > \mathcal \epsilon_2$.  To construct the nets, it suffices to show that for all $i = (\mathcal G, \mathcal H, \epsilon) \in I$, there are functions $u \colon G \rightarrow U(\mathcal Q)$ and $v \colon H \rightarrow U(\mathcal Q)$ such that for all $g_1, g_2 \in \mathcal G$, $g \in \mathcal G \setminus \{1\}$, $h_1, h_2 \in \mathcal H$, $h \in \mathcal H \setminus \{1\}$, and $n \in \mathcal G \cap \mathcal H$,
	\begin{enumerate}
		\item $u_{g_1 g_2} = u_{g_1} u_{g_2}$ and $v_{h_1 h_2} = v_{h_1} v_{h_2}$,
		\item $\mathrm{tr}_\mathcal Q(u_g) = 0$ and  $\mathrm{tr}_\mathcal Q(v_h) = 0$, and
		\item $\|u_n - v_n\| < \epsilon$.
	\end{enumerate}
	
	After replacing $G$ and $H$ with the subgroups generated by $\mathcal G$ and $\mathcal H$, we may assume $G$ and $H$ are countable.  
	By Theorem~\ref{thm:classification}, there are group homomorphisms $u \colon G \rightarrow U(\mathcal Q)$ and $v' \colon H \rightarrow U(\mathcal Q)$ such that $\mathrm{tr}_\mathcal Q(u_g) = 0$ and $\mathrm{tr}_\mathcal Q(v'_h) = 0$ for $g \in G \setminus \{1\}$ and $h \in H \setminus \{1\}$.  By another application of Theorem~\ref{thm:classification}, $u|_N$ and $v'|_N$ are approximately unitarily equivalent.  Fix $w \in  U(\mathcal Q)$ such that for all $n \in \mathcal G \cap \mathcal H$,
	\begin{equation}
		\| u_n - wv'_nw^* \| < \epsilon.
	\end{equation}
	Then $u$ and $v = \mathrm{ad}(w) \circ v'$ satisfy the required properties.
\end{proof}

All that remains in the proof of Theorem~\ref{thm:main} is to patch together the approximate representations from the previous lemma to obtain MF approximations for the amalgamated free product.  While it is easy to obtain approximate finite dimensional unitary representations of $G *_N H$ from Lemma~\ref{lem:nets}, arranging conditions~\ref{def:MF-group2} and~\ref{def:MF-group3} of Definition~\ref{def:MF-group} is harder.  This will be done with the aid of Lemma~\ref{lemma}.

\begin{proof}[Proof of Theorem \ref{thm:main}]
	Note that $N$ is a normal subgroup of $G *_N H$ and there is a canonical isomorphism  $(G *_N H)/N \cong (G/N) * (H/N)$.  As $G / N$ and $H / N$ are amenable, they are exact.  By the main result of \cite{Dykema99}, free products of exact groups are exact, so $(G *_N H)/N$ is exact.  Also, $G/N$ and $H/N$ are MF by Theorem~\ref{thm:qd}, and hence $(G *_N H)/N$ is MF by Theorem~\ref{thm:free}.
	
	Let $(u^i \colon G \rightarrow U(\mathcal Q))_{i \in I}$ and $(v^i \colon H \rightarrow U(\mathcal Q))_{i \in I}$ be nets as in Lemma~\ref{lem:nets}.  Fix an ultrafilter $\omega$ on $I$ such that for all $i_0 \in I$, we have $\{i \in I : i \geq i_0\} \in \omega$.  Let $\mathcal Q_\omega$ denote the corresponding norm ultrapower of $\mathcal Q$ and let $\mathrm{tr}_{\mathcal Q_\omega}$ denote the trace on $\mathcal Q_\omega$ induced by $\mathrm{tr}_\mathcal Q$.  Then $\mathrm{tr}_{\mathcal Q_\omega}$ is MF.
	
	The nets $(u^i)_{i \in I}$ and $(v^i)_{i \in I}$ define group homomorphisms $u \colon G \rightarrow U(\mathcal Q_\omega)$ and $v \colon H \rightarrow U(\mathcal Q_\omega)$ with 
	$\mathrm{tr}_{\mathcal Q_\omega}(u_g) = 0$ and $\mathrm{tr}_{\mathcal Q_\omega}(v_h) = 0$ for all $g \in G \setminus\{1\}$ and $h \in H \setminus \{1\}$ and with $u|_N = v|_N$.  Then there is a group homomorphism $w \colon G *_N H \rightarrow U(\mathcal Q_\omega)$ such that $w|_G = u$ and $w|_H = v$.  In particular, $\mathrm{tr}_{\mathcal Q_\omega}(w_n) = 0$ for all $n \in N \setminus\{1\}$, and the result follows from Lemma~\ref{lemma}.
\end{proof}

\begin{remark}
If $(G_j)_{j \in J}$ is a (possibly infinite) collection of amenable groups with a common normal subgroup $N$, then the amalgamated product $\Asterisk_N\, G_j$ is MF.  The proof reduces to the case $J$ is finite using that MF is a local property, and then the proof is essentially same except Lemma~\ref{lem:nets} must be modified to produce several nets (one for each $j \in J$) with the analogous properties.
\end{remark}

\bibliographystyle{plain}

\begin{thebibliography}{10}
	
	\bibitem{Avitzour82}
	Daniel Avitzour.
	\newblock Free products of {$C^{\ast} $}-algebras.
	\newblock {\em Trans.\ Amer.\ Math.\ Soc.}, 271(2):423--435, 1982.
	
	\bibitem{Bekka-Cherix-Valette93}
	M.{\ }E.{\ }B. Bekka, P.-A. Cherix, and A.~Valette.
	\newblock Proper affine isometric actions of amenable groups.
	\newblock In {\em Novikov conjectures, index theorems and rigidity, {V}ol. 2
		({O}berwolfach, 1993)}, volume 227 of {\em London Math.\ Soc.\ Lecture Note
		Ser.}, pages 1--4. Cambridge Univ. Press, Cambridge, 1995.
	
	\bibitem{Blackadar-Kirchberg97}
	Bruce Blackadar and Eberhard Kirchberg.
	\newblock Generalized inductive limits of finite-dimensional {$C^*$}-algebras.
	\newblock {\em Math. Ann.}, 307(3):343--380, 1997.
	
	\bibitem{Brown06}
	Nathanial~P. Brown.
	\newblock Invariant means and finite representation theory of {$C^*$}-algebras.
	\newblock {\em Mem. Amer. Math. Soc.}, 184(865):viii+105, 2006.
	
	\bibitem{Brown-Dykema-Jung08}
	Nathanial~P. Brown, Kenneth~J. Dykema, and Kenley Jung.
	\newblock Free entropy dimension in amalgamated free products.
	\newblock {\em Proc. Lond. Math. Soc. (3)}, 97(2):339--367, 2008.
	\newblock With an appendix by Wolfgang L\"{u}ck.
	
	\bibitem{Carrion-Dadarlat-Eckhardt13}
	Jos\'{e}~R. Carri\'{o}n, Marius Dadarlat, and Caleb Eckhardt.
	\newblock On groups with quasidiagonal {$C^*$}-algebras.
	\newblock {\em J. Funct. Anal.}, 265(1):135--152, 2013.
	
	\bibitem{Collins-Male14}
	Beno\^{i}t Collins and Camille Male.
	\newblock The strong asymptotic freeness of {H}aar and deterministic matrices.
	\newblock {\em Ann. Sci. \'{E}c. Norm. Sup\'{e}r. (4)}, 47(1):147--163, 2014.
	
	\bibitem{Connes76}
	A.~Connes.
	\newblock Classification of injective factors. {C}ases $\rm{II}_1$,
	$\rm{II}_\infty$, $\rm{III}_\lambda$, $\lambda \neq 1$.
	\newblock {\em Ann. of Math. (2)}, 104(1):73--115, 1976.
	
	\bibitem{Dykema99}
	Kenneth~J. Dykema.
	\newblock Free products of exact groups.
	\newblock In {\em {$C^*$}-algebras ({M}\"{u}nster, 1999)}, pages 61--70.
	Springer, Berlin, 2000.
	
	\bibitem{Elek-Szabo11}
	G\'{a}bor Elek and Endre Szab\'{o}.
	\newblock Sofic representations of amenable groups.
	\newblock {\em Proc. Amer. Math. Soc.}, 139(12):4285--4291, 2011.
	
	\bibitem{Haagerup-Thorbjornsen05}
	Uffe Haagerup and Steen Thorbj{\o}rnsen.
	\newblock A new application of random matrices: {${\rm Ext}(C^*_{\rm
			red}(F_2))$} is not a group.
	\newblock {\em Ann. of Math. (2)}, 162(2):711--775, 2005.
	
	\bibitem{Hayes15}
	Ben Hayes.
	\newblock An {$l^p$}-version of von {N}eumann dimension for {B}anach space
	representations of sofic groups {II}.
	\newblock {\em J. Funct. Anal.}, 269(8):2365--2426, 2015.
	
	\bibitem{JNVWY}
	Zhengfeng Ji, Anand Natarajan, Thomas Vidick, John Wright, and Henry Yuen.
	\newblock {MIP}$^*$={RE}, 2020.
	\newblock arXiv:2001.04383 ([quant-ph]).
	
	\bibitem{Louder-Magee22}
	Larsen Louder and Michael Magee.
	\newblock Strongly convergent unitary representations of limit groups, with an
	appendix by will hide and michael magee, 2022.
	\newblock arXiv:2210.08953 [math.GR].
	
	\bibitem{Magee-Thomas23}
	Michael Magee and Joe Thomas.
	\newblock Strongly convergent unitary representations of right-angled {A}rtin
	groups.
	\newblock arXiv:2308.00863 [math.GR], 2023.
	
	\bibitem{Murray-vonNeumann36}
	F.~J. Murray and J.~von Neumann.
	\newblock On rings of operators.
	\newblock {\em Ann. of Math. (2)}, 37(1):116--229, 1936.
	
	\bibitem{Ozawa04}
	Narutaka Ozawa.
	\newblock About the {QWEP} conjecture.
	\newblock {\em Internat.\ J.\ Math.}, 15(5):501--530, 2004.
	
	\bibitem{Ozawa13}
	Narutaka Ozawa.
	\newblock Dixmier approximation and symmetric amenability for {$C^*$}-algebras.
	\newblock {\em J. Math. Sci. Univ. Tokyo}, 20(3):349--374, 2013.
	
	\bibitem{Pestov08}
	Vladimir~G. Pestov.
	\newblock Hyperlinear and sofic groups: a brief guide.
	\newblock {\em Bull. Symbolic Logic}, 14(4):449--480, 2008.
	
	\bibitem{Paunescu11}
	Liviu P\u{a}unescu.
	\newblock On sofic actions and equivalence relations.
	\newblock {\em J. Funct. Anal.}, 261(9):2461--2485, 2011.
	
	\bibitem{Rainone-Schafhauser19}
	Timothy Rainone and Christopher Schafhauser.
	\newblock Crossed products of nuclear {$C^*$}-algebras and their traces.
	\newblock {\em Adv. Math.}, 347:105--149, 2019.
	
	\bibitem{Schafhauser20}
	Christopher Schafhauser.
	\newblock Subalgebras of simple {AF}-algebras.
	\newblock {\em Ann. of Math. (2)}, 192(2):309--352, 2020.
	
	\bibitem{Tikuisis-White-Winter17}
	Aaron Tikuisis, Stuart White, and Wilhelm Winter.
	\newblock Quasidiagonality of nuclear {$C^*$}-algebras.
	\newblock {\em Ann. of Math. (2)}, 185(1):229--284, 2017.
	
	\bibitem{Tu99}
	Jean-Louis Tu.
	\newblock La conjecture de {B}aum-{C}onnes pour les feuilletages moyennables.
	\newblock {\em $K$-Theory}, 17(3):215--264, 1999.
	
	\bibitem{Voiculescu91}
	Dan Voiculescu.
	\newblock Limit laws for random matrices and free products.
	\newblock {\em Invent. Math.}, 104(1):201--220, 1991.
	
\end{thebibliography}

\end{document}